\documentclass[11pt,draft]{amsart}
\usepackage{amssymb, amstext, amscd, amsmath}
\usepackage[all]{xy}

\makeatletter
\def\@cite#1#2{{\m@th\upshape\bfseries%
[{#1\if@tempswa{\m@th\upshape\mdseries, #2}\fi}]}}
\makeatother
\theoremstyle{plain}
\newtheorem{theorem}{Theorem}[section]
\newtheorem{corollary}[theorem]{Corollary}
\newtheorem{proposition}[theorem]{Proposition}
\newtheorem{lemma}[theorem]{Lemma}

\theoremstyle{definition}
\newtheorem{definition}[theorem]{Definition}
\newtheorem{example}[theorem]{Example}
\newtheorem{remark}[theorem]{Remark}
\theoremstyle{remark}

\newcommand{\bbC}{{\mathbb{C}}}

\newcommand{\bbN}{{\mathbb{N}}}
\newcommand{\bbQ}{{\mathbb{Q}}}
\newcommand{\bbR}{{\mathbb{R}}}
\newcommand{\bbT}{{\mathbb{T}}}
\newcommand{\bbZ}{{\mathbb{Z}}}

\newcommand{\A}{{\mathcal{A}}}
\newcommand{\B}{{\mathcal{B}}}
\newcommand{\C}{{\mathcal{C}}}
\newcommand{\D}{{\mathcal{D}}}

\newcommand{\F}{{\mathcal{F}}}

\renewcommand{\L}{{\mathcal{L}}}

\newcommand{\R}{{\mathcal{R}}}
\renewcommand{\S}{{\mathcal{S}}}
\newcommand{\T}{{\mathcal{T}}}

\newcommand{\fA}{{\mathfrak{A}}}

\newcommand{\fD}{{\mathfrak{D}}}

\newcommand{\fG}{{\mathfrak{G}}}

\renewcommand{\phi}{\varphi}
\newcommand{\upchi}{{\raise.35ex\hbox{\ensuremath{\chi}}}}

\def\gl{\lambda}

\newcommand{\alg}{\operatorname{alg}}

\newcommand{\dist}{\operatorname{dist}}

\newcommand{\lat}{\operatorname{lat}}

\newcommand\supp{\mathop{\rm supp}}

\newcommand{\ca}{\mathrm{C}^*}
\newcommand{\cenv}{\mathrm{C}^*_{\text{env}}}
\newcommand{\Gr}{G_{\text{red}}}

\newcommand{\hdi}{\hat{\delta}}

\newcommand{\sca}[1]{\left\langle#1\right\rangle}

\newcommand{\nor}[1]{\left\Vert #1\right\Vert}

\begin{document}

\title[Integer valued cocycles]{Limit algebras and integer-valued cocycles, revisited}

\author[E.G. Katsoulis]{Elias~G.~Katsoulis}
\address {Department of Mathematics
\\East Carolina University\\ Greenville, NC 27858\\USA}
\email{katsoulise@ecu.edu}

\author[C. Ramsey]{Christopher~Ramsey}
\address {Department of Mathematics
\\University of Virginia\\ Charlottesville, VA 22904\\USA}
\email{cir6d@virginia.edu}

\thanks{2010 {\it  Mathematics Subject Classification.}
46L08, 47B49, 47L40, 47L65}
\thanks{{\it Key words and phrases:} C$^*$-correspondence, cocycle, tensor algebra, triangular limit algebra}

\maketitle

\begin{abstract}
A triangular limit algebra $\A$ is isometrically isomorphic to the tensor algebra of a $\ca$-correspondence if and only if its fundamental relation $\R(\A)$ is a tree admitting a $\bbZ^+_0$-valued continuous and coherent cocycle. For triangular limit algebras which are isomorphic to tensor algebras, we give a very concrete description for their defining $\ca$-correspondence and we show that it forms a complete invariant for isometric isomorphisms between such algebras. A related class of operator algebras is also classified using a variant of the Aho-Hopcroft-Ullman algorithm from computer aided graph theory.
\end{abstract}

\section{Introduction}
In this paper we explore the common ground between two important classes of operator algebras: triangular limit algebras and tensor algebras of $\ca$-correspondences. The limit algebras were introduced in the 80's by Power \cite{Pow1} and Peters, Poon and Wagner \cite{PPW1} and in a broader context by Muhly and Solel \cite{MScrel}. These algebras were studied extensively in the 90's by many authors, including \cite{DavKatAdv, Don, DonH, DonHJFA, Hu, LS, PPW1, PPWcoc,Pow, Powlex}. The tensor algebras of $\ca$-correspondences form a newer class of algebras which continues to be at the cutting edge of scientific inquiry. This class of operator algebras incorporates as examples many classes of operator algebras that have been studied independently in the past, including semicrossed products of $\ca$-algebras~\cite{Pet}, quiver (or graph) algebras~\cite{MS} and non-commutative disc algebras~\cite{Pop}, to mention just a few. The tensor algebras were introduced by Muhly and Solel in \cite{MS} soon after Pimsner's seminal paper \cite{Pim}. Contributions here include the work of various authors as well \cite{DavKatMem, DavRoyd,  Kak, KakKatTrans, KatsoulisKribsJFA, MSten, Pop, Popten}.

Our primary objective is to characterize which operator algebras belong simultaneously to both classes mentioned above. In the finite dimensions, a limit algebra is just a digraph algebra; a triangular digraph algebra is isometrically isomorphic to the tensor algebra of a $\ca$-correspondence if and only if its digraph is the transitive completion of an out-forest. (See below for definitions.) In the infinite dimensions a useful answer has to be more involved and requires the concept of the fundamental relation. This is achieved in Theorem~\ref{main} where we show that a triangular limit algebra $\A$ is isometrically isomorphic to the tensor algebra of a $\ca$-correspondence if and only if its fundamental relation $\R(\A)$ is a tree admitting a $\bbZ^+_0$-valued continuous and coherent cocycle. 

Theorem~\ref{main} relates to earlier work of others and has some interesting applications. The concept of a tree semigroupoid first appeared in \cite{MSMem}, in an important study of a related class of operator algebras, the so called tree algebras. (Following \cite{DKDoc}, the tree algebras will be called semi-Dirichlet in this paper.)  Cocycles on non-selfadjoint operator algebras have been studied before in a slightly different context \cite{PPWcoc, PPWex, Sol} and our results enrich these studies.  Using Theorem~\ref{main} we do show that the triangular limit algebras which are also tensor algebras are exactly the direct limits of their finite dimensional counterparts, under appropriate embeddings (see Theorem~\ref{presentation1}). This puts an old result of Poon and Wagner \cite[Theorem 2.9]{PW} in a new perspective. In Proposition~\ref{counterexample}, we give examples of Dirichlet algebras which are not \textit{isometrically} isomorphic to tensor algebras, thus strengthening a recent result of Kakariadis \cite{Kak}. Also Proposition \ref{secondex} gives new examples of semi-Dirichlet algebras which are neither Dirichlet nor tensor algebras. This continues a theme originating from \cite{KR} and provides (yet another) class of examples that answer a question of Ken Davidson.

In the last section of the paper we classify the various algebras appearing in this paper. It is still an open problem whether or not the unitary equivalence class of a $\ca$-correspondence forms a complete invariant for isometric isomorphisms between the associated tensor algebras. Nevertheless, for the class of tensor algebras considered here, the answer is affirmative as Corollary~\ref{clas1} shows. This is accomplished by showing that the defining $\ca$- correspondence for a tensor algebra $\A$, which is isometrically isomorphic to a triangular limit algebra, can actually be materialized as an appropriate invariant subset of its fundamental relation $\R(\A)$ (Theorem~\ref{maindescr}). The result then follows from a well-known result of Power \cite[Theorem 7.5]{Pow}. In the same section, we also classify another class of limit algebras, the so-called tree-refinement algebras, which are semi-Dirichlet but not isomorphic to tensor algebras. These algebras generalize the familiar refinement algebras and interestingly their classification requires a variant of the Aho-Hopcroft-Ullman algorithm \cite{AHU} from computer aided graph theory.

 \section{preliminaries}

Let $D$ be a $\ca$-algebra. An \emph{inner-product right $D$-module} is a linear space $X$ which is a right $D$-module together with a $D$-valued inner product $\sca{\cdot,\cdot}$ that satisfies
\begin{align*}
\sca{\xi, \gl \zeta+ \eta}&= \gl\sca{\xi,\zeta} + \sca{\xi,\eta}\qquad
\\
\sca{\xi, \eta d}&= \sca{\xi,\eta}d \qquad \\
\sca{\eta,\xi}&= \sca{\xi,\eta}^* \qquad \\
\sca{\xi,\xi} &\geq 0 \text{; if }\sca{\xi,\xi}= 0 \text{ then }
\xi=0,
\end{align*}
with $\zeta, \eta, \xi \in X$, $\lambda \in \bbC$ and $d \in D$. For $\xi \in X$ we write $\nor{\xi}_X^2 := \nor{\sca{\xi,\xi}}_{D}$ and one can deduce that $\nor{\cdot}_X$ is actually a norm. Equipped with that norm, $X$ will be called a \emph{ Hilbert $D$-module} if it is complete and will be denoted as $X_{D}$, or simply $X$.

For a Hilbert $D$-module $X$ we define the set $L(X)$ of the \emph{adjointable maps} that consists of all maps $s:X \rightarrow X$ for which there is a map $s^*: X \rightarrow X$ such that
\[
\sca{s\xi,\eta}= \sca{\xi,s^*\eta}, \,\, \xi, \eta \in X.
\]

The compact operators $K(X) \subseteq L(X)$ is the closed subalgebra of $L(X)$ generated by the ``rank one" operators 
\[
\theta_{\xi , \eta}(z) := \xi  \sca{\eta ,z}, \quad \xi, \eta ,z \in X
\]

\begin{definition} A $\ca$-correspondence $(X, D, \phi)$ consists of a Hilbert D-module $(X, D)$ and a left action
\[
\phi: D \longrightarrow L(X).
\]
If $\phi$ is injective then the $\ca$-correspondence $(X, D, \phi)$ is said to be injective.\end{definition}

A (Toeplitz) representation $(\pi,t)$ of $X$ into a $\ca$-algebra $\B$, is a pair consisting of a $*$-homomorphism $\pi\colon D \rightarrow B$ and a linear map $t\colon X \rightarrow \B$, such that
\begin{enumerate}
 \item $\pi(d)t(\xi)=t(\phi_X(d)(\xi))$,
 \item $t(\xi)^*t(\eta)=\pi(\sca{\xi,\eta}_X)$,
\end{enumerate}
for $d \in D$ and $\xi,\eta\in X$. An easy application of the $\ca$-identity shows that
\begin{itemize}
\item[(iii)] $t(\xi)\pi(d)=t(\xi d)$
\end{itemize}
 is also valid. A representation $(\pi , t)$ is said to be \textit{injective} iff $\pi$ is injective. In that case $t$ is an isometry.

\begin{definition} The \emph{tensor algebra} $\T_{X}^+$ is the norm-closed subalgebra of $\T_X$ generated by all elements of the form $\pi_{\infty}(a), t_{\infty}(\xi)$, $a \in A$, $\xi \in X$, where $(\pi_{\infty}, t_{\infty})$ denotes the universal Toeplitz representation of $(X, A, \phi)$.
\end{definition}

A special class of $\ca$-corespondences arises from the class of finite directed graphs. Let $G = (G^0, G^1, r, s)$ be a finite graph, where $G^0, G^1$ denotes the vertex and edge sets respectively and $r,s: G^1\rightarrow G^0$ the range and source maps respectively. Let $D= c_0( G^0)$, $ X_{G} = c_0(G^1)$
\[
\sca{\xi , \eta}(p)=\sum_{s(e)=p} \, \overline{\xi(e)}\eta (e), \,\, \xi, \eta \in X_{G},  \, p\in G^0
\]
and $\big( \phi( f) \xi g \big)(e) = f(r(e))\xi(e)h(s(e))$, with $ \xi \in X_{G}$, $f,g \in D$ and $ e \in G^1$.

Let $(\pi, t)$ be a non-degenerate Toeplitz representation of the \textit{graph correspondence} $(X_{G} , D)$. Let $L_p=\pi(1_p)$, $p \in G^0$ and $L_e = t(1_e)$, $e \in G^1$, where $1_e \in X_{G}$ denotes the characteristic function of the singleton $\{e\}$, $e \in G^1$ and similarly for $1_g \in D$, with $g \in G^0$. Then $L_pL_q=\pi(1_p1_q)=\delta_{p,q}L_p$, where $p,q \in G^0$.
Also, 
\[
L_e^*L_f=t(1_e)^*t(1_f)= \pi(\sca{1_e, 1_f})=0, \,\, \mbox{for }e, f \in G^1, e\neq f
\]
and similarly $L_e^*L_e=L_{s(e)}$, $e \in G^1$. It is clear that the family $\{ L_e\}_{ e\in
G^{1}}$  of partial isometries  and  $\{ L_p \}_{p\in G^{0}}$ of projections obey
the Cuntz-Krieger-Toeplitz relations
\begin{equation} \label{CKT}
  \left\{
\begin{array}{lll}
(1)  & L_p L_q = 0 & \mbox{$\forall\, p,q \in G^{0}$, $ p \neq q$}  \\
(2) & L_{e}^{*}L_f = 0 & \mbox{$\forall\, e, f \in G^{1}$, $e \neq f $}  \\
(3) & L_{e}^{*}L_e = L_{s(e)} & \mbox{$\forall\, e \in G^{1}$}      \\
(4)  & L_e L_{e}^{*} \leq L_{r(e)} & \mbox{$\forall\, e \in G^{1}$} \\
(5)  & \sum_{r(e)=p}\, L_e L_{e}^{*} \leq L_{p} & \mbox{$\forall\, p
\in G^{1}$ }
\end{array}
\right.
\end{equation}

Conversely given  a family $\{ L_e\}_{ e\in
G^{1}}$  of partial isometries  and  $\{ L_p \}_{p\in G^{0}}$ of projections obeying
the Cuntz-Krieger-Toeplitz relations of (\ref{CKT}), we can define a Toeplitz representation of the graph correspondence $(X_{G}, D)$ by setting $\pi(1_p)=L_p$, $ p \in G^0$, $t(1_e)= L_e$, $e \in G^1$ and extending by linearity. In this case the tensor algebra of $(X_{G}, D)$ is the familiar quiver algebra $\T_{G}^+$ of Muhly and Solel \cite{MS}. 

A maximal abelian selfadjoint subalgebra (masa) $\D$ of an AF $\ca$-algebra $\C$ is said to be \textit{regular canonical} provided that there exists a nested sequence of finite dimensional $\ca$-algebras $\{ \C_n\}_{in= 1}^{\infty}$ so that the following are satisfied
\begin{itemize}
\item[(i)] $\C = \overline{\cup_n \C_n}$
\item[(ii)] $\D_n= \C_n \cap \D$ is a masa in $\C_n$, for all $n \in \bbN$, and $\D = \overline{\cup_n \D_n}$
\item[(iii)] $N_{\D_n}(\C_n) \subseteq N_{\D_{n+1}}(\C_{n+1})$, for all $n \in \bbN$, 
\end{itemize}
where
\[
N_Y(X) = \{x \in X\mid (x^*x)^2=x^*x \mbox{ and } xyx^*, x^*yx \in Y, \mbox{ for all }y \in Y \}.
\]
Alternatively such masas can be described as follows. Consider an AF $\ca$-algebra $\C$ as a direct limit $\C= \varinjlim (\C_n, \rho_n)$, where each $\C_n$ is direct sum of full matrix algebras and all the embeddings $\rho_n$ are regular, i.e., they map matrix units to sums of matrix units. Then the subalgebra $\D \subseteq \C$ determined by the (sub)system $\D= \varinjlim (\D_n, \rho_n)$, where each $\D_n \subseteq \C_n$ consists of the diagonal matrices, is a regular canonical masa. 

In the sequel we will work exclusively with operator algebras $\A \subseteq \C$ which contain a canonical masa, i.e., regular canonical subalgebras of AF $\ca$-algebras. We call such algebras \textit{limit algebras}. By inductivity \cite[Theorem 4.7]{Pow} this implies that for such an algebra $\A$, each of the subalgebras $\C_n \cap \A$ is a finite dimensional CSL algebras containing $\D_n$ as a masa and the inclusion 
\[
N_{\D_n}(\A_n) \subseteq N_{\D_{n+1}}(\A_{n+1})
\]
is satisfied for all $n \in \bbN$. Alternatively such algebras can be described as direct limit algebras of digraph algebras, i.e., subalgebras of the $k \times k$ matrices over $\bbC$ containing the diagonal matrices. This is done as follows. Let  $\D= \varinjlim (\D_n, \rho_n)$ be a canonical masa of an AF $\ca$-algebra $\C= \varinjlim (\C_n, \rho_n)$ as in the previous paragraph, and let  all $\{ \A_n\}_{n=1}^{\infty}$ be digraph subalgebras of $\C_n$ with $\rho(\A_n) \subseteq \A_{n+1}$, for all $n \in \bbN$. Then $\A= \varinjlim (\A_n, \rho_n)$ is a regular limit algebra and any regular limit algebra arises this way. If $\A$ is a regular canonical subalgebra of an AF $\ca$-algebra, then the quadruple $\big\{ \A_n, \C_n, \D_n, \rho_n\big\}_{n=1}^{\infty}$ is said to be a \textit{presentation} for $\A$. 
Recall that an operator algebra $\A$ contained in a C$^*$-algebra $\C$ is called \textit{triangular} if $\A \cap \A^* = \D$ is abelian (and a masa in this context). 

A special but important class of triangular limit algebras are the \textit{strongly maximal} TAF algebras (for triangular AF). These are limit algebras of the form $\A = \varinjlim (\A_n, \rho_n)$ as in the previous paragraph, with the extra requirement that each $\A_n$ consists of direct sums of upper triangular matrices.

\begin{example} \label{standard}
Let $\{e_{ij}\}_{i,j=1}^{n}$ denote the usual matrix unit system of the algebra $M_n(\bbC)$ of $n \times n$ complex matrices. An embedding $\sigma:M_n(\bbC) \rightarrow M_{mn}(\bbC)$ is said to be \textit{standard} if it satisfies $\sigma(e_{ij})=\sum_{k=0}^{m-1} e_{i +kn, j+kn}$, for all $i,j$. If $\A = \varinjlim (\A_n, \rho_n)$ is a limit algebra with all the embeddings $\rho_n: \A_n \rightarrow \A_{n+1}$, $n \in \bbN$, being direct sums of standard embeddings, then $\A$ is said to be a \textit{standard limit algebra}.
\end{example}

Finally some terminology from graph theory. A finite (weakly) connected graph $G$ is said to be an \textit{out-tree} (or \textit{arboresence}) iff for a vertex $p$ (called the \textit{root}) and any other vertex $q$, there is exactly one directed path from $p$ to $q$. In other words, an out-tree is a finite directed graph where no vertex receives more than one edge and its underlying undirected graph is both connected and acyclic. An \textit{out-forest} is a finite disjoint union of out-trees.

\section{Identifying which triangular limit algebras are isomorphic to tensor algebras}

Let $\A$ be a limit algebra with presentation $\big\{ \A_n, \C_n, \D_n, \rho_n\big\}_{n=1}^{\infty}$ and let $\D^{\star}$ be the Gelfand spectrum of its canonical masa $\D =  \varinjlim (\D_n, \rho_n)$. If $p \in \D$ is a projection, let $\tilde{p}=\{ x \in \D^{\star}\mid x(p)=1\}$. Each matrix unit $e \in \A_n$, $ n \in \bbN$, induces a partial homeomorphism $$\widetilde{\phantom{*}ee^*}\ni x(\cdot) \longmapsto x(e\cdot e^*)\in \widetilde{e^*e}.$$ Let $\hat{e} \subseteq \D^{\star} \times \D^{\star}$ denote the graph of that partial homeomorphism. We define
\[
\R(\A)= \bigcup\big\{ \hat{e} \mid e \mbox{ is a matrix unit in } \cup_n \A_n\big\}.
\]
We topologize $\R(\A)$ by using as a basis of open sets all sets of the form $\hat{e}$, where $e$ is ranging over all matrix units in $ \cup_n \A_n$. We call the AF-semigroupoid $\R(\A)$ the \textit{spectrum} or \textit{fundamental relation} of $\A$.

\begin{definition} \label{cocycle}
Let $\A$ be a triangular limit algebra and let $\R(\A)$ be the associated AF-semigroupoid on the Gelfand space $\D^{\star}$ of its canonical masa $\D$. A real valued function $\delta  : \R(\A) \rightarrow \bbR$ is said to be a \textit{cocycle} if it satisfies the following
\begin{itemize}
\item[(i)] $\delta(x, z) =\delta(x,y) +\delta(y, z)$, for all $(x, y) , (y, z) \in R$ 
\item[(ii)] $\delta^{-1}(\{0\}) = \widehat{\D}\equiv \cup_{e \in \D}\, \hat{e}$.
\end{itemize}
 If $\delta\big(\R(\A)\big) \subseteq \bbZ^+_0=\{0,1,2, \dots\}$ then $\delta $ is said to be a $\bbZ^+_0$-valued cocycle. 
An integer valued cocycle $\delta$ on $\R(\A)$ is said to be \textit{coherent} if any element in $\delta ^{-1}(\{ k\})$, $k=2, 3, \dots$, can be written as the product of $k$ elements from $\delta ^{-1}(\{ 1\})$.
\end{definition}

The concept of a cocycle has been studied before in the context of limit algebras in a different but nevertheless related setting \cite{PPoon, PPWcoc, PPWex}. These authors considered continuous $\bbZ$-valued cocycles defined on the AF-groupoid of the enveloping $\ca$-algebra, satisfying only property (i) in Definition~\ref{cocycle} and the additional property that the inverse image of the non-negative values of the cocycle coincides with the semigroupoid of the limit (sub)algebra. For strongly maximal TAF algebras, these two concepts of cocycle coincide.

\begin{proposition} Let $\A = \varinjlim (\A_n, \rho_n)$ be a strongly maximal TAF algebra with enveloping $\ca$-algebra  $\C = \varinjlim (\C_n, \rho_n)$ and diagonal $\D$. Let $\R(\C)$ and $\R(\A)$ be the AF-(semi)groupoids associated with $\C$ and $\A$ respectively.  If $\delta : \R(\A) \rightarrow \bbZ^+_0$ is a continuous coherent cocycle then $\delta $ extends to a continuous coherent cocycle $\tilde{\delta}: \R(\C) \rightarrow \bbZ$ with $\tilde{\delta}^{-1}( [0, \infty ) ) = \widehat{\A}$.
\end{proposition}

\begin{proof}Since $\A$ is strongly maximal then given any $(x, y) \in \R(\C)$, either $(x, y) \in \R(\A)$ or otherwise $(y,x) \in \R(\A)$. In the first case set $\tilde{\delta}(x, y) = \delta(x,y)$ or otherwise $\tilde{\delta}(y, x) = -\delta(x,y)$. It is easy to see that this defines the desired cocycle $\tilde{\delta}: \R(\C) \rightarrow \bbZ$.
\end{proof}

Even though there is no distinguishing between the two concepts of cocycle in the strongly maximal TAF case, the situation changes dramatically if one removes maximality.

\begin{example}
Consider the out-tree $G$ appearing below

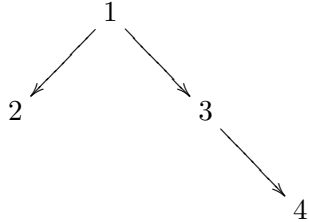
\vspace{.05in}

\[
\xymatrix{& 1 \ar[dl] \ar[dr]& &\\ 
2 & &3 \ar[dr] & \\
&& & 4}
\]

\vspace{.05in}
\noindent and let $\A$ be the digraph algebra corresponding to the transitive completion of $G$. (Note also that $\A \simeq \T^+_G$.) It is immediate that the only coherent cocycle $d: R(\A) \rightarrow \bbZ^+_0$ satisfies $d(\widehat{e}_{31})=d(\widehat{e}_{43})= d(\widehat{e}_{21})= \{1\}$. Therefore, any cocycle $\tilde{\delta}: \R\big(\ca(\A)\big) \rightarrow \bbZ$ extending $d$ would satisfy $d(\widehat{e}_{42})=\{1\}$ and $d(\widehat{e}_{32})=\{ 0\}$. Clearly it would be impossible from such a cocycle to recapture either $\D$ or $\A$ since $\D \neq \tilde{\delta}^{-1}(\{0\})$ and $\A \neq \tilde{\delta}^{-1}([0, \infty))$ 
\end{example}

Having ascertained the necessity of using cocycles defined only on the semigroupoid of the limit algebra we now have.

\begin{theorem} \label{main1}
Let $\A$ be a triangular limit algebra, i.e., triangular regular canonical subalgebra of an AF $\ca$-algebra. If $\A$ is isometrically isomorphic to the tensor algebra of a $\ca$-correspondence, then $\R(\A)$ admits a \textup{(}necessarily unique\textup{)} $\bbZ^+_0$-valued continuous and coherent cocycle. 
\end{theorem}

\begin{proof}
Assume that $\A = \varinjlim (\A_n, \rho_n)$ is isometrically isomorphic to the tensor algebra $\T_X^+$ of a $\ca$-correspondence $(X, D)$. Now the fixed point algebra of the natural gauge action $\{ \psi_z\}_{z \in \bbT}$ on $\T_{X}^+$ equals the diagonal $(\T_X^+)^*\cap \T_X^+ =D$ and $X^{n} =X\cdot X\cdot \ldots \cdot X\subseteq \T_{X}^+$ coincides with the set of all $x  \in \T_{X}^+$ with $\psi_{z}(x) = z^n x$, $z \in \bbT$. Furthermore a standard trick with the Fejer kernel implies that the linear space generated by $D$ and all $X^{ n}$, $n\in \bbN$, is dense in $\T_{X}^+$.

Since $\A$ is isometrically isomorphic to $\T_X^+$, it inherits the action $\{ \psi_z\}_{z \in \bbT}$. And since an isometric isomorphism maps diagonals to diagonals, the fixed point algebra of that action is $\D$. We may assume that $X\subseteq \A$ and the left action on $X$ comes from the diagonal $\D$. Finally, the linear space generated by $\D$ and all $X^{ n}$, $n\in \bbN$, is dense in $\A$.

Let us say that a matrix unit $e \in \A$ is \textit{$k$-graded} if $\psi_{z}(e) = z^k e$, for some $k \in \bbZ^+_0$ and all $ z \in \bbT$. Let $\F(X)$ denote the linear space generated in $\A$ by all finitely graded matrix units.

\vspace{.05in}
\noindent \textit{Claim 1:} All matrix units in $\A$ belong to $\F(X)$.

\vspace{.05in}
\noindent Assume instead that there is a matrix unit $e \in \A_n $ with $e \notin \F(X) \cap \A_n $ and so $e \notin \F(X) \cap \A_m $, for all $m \geq n$. Since each $\F(X) \cap \A_m $ is a $\D_m$-module, $\dist(e,  \F(X) \cap \A_m) = 1$ and so
\begin{equation} \label{contradiction}
\dist(e,  \F(X)) = 1.
\end{equation}

On the other hand $X$ is inductive in $\A$ since it is a $\D$-bimodule. Thus, by this inductivity and the fact that the embeddings are regular $X$ is generated as a normed space by its matrix units. In the same way for each $n\geq 1$, $X^{n}$ is a $\D$-bimodule and again is generated by the matrix units it contains. Therefore linear combinations of finitely graded matrix units are dense, i.e., $\F(X)$ is dense in $\A$. But this contradicts (\ref{contradiction}) and proves the claim.

\vspace{.05in}

We will now use the claim in order to build a cocycle $\delta : \R(\A) \rightarrow \bbZ^+_0$. If $(x,y) \in \R(\A)$ then by definition there exists a matrix unit $e$ so that $(x, y)$ belongs to its graph $\hat{e}$. By the claim above, there exist finitely graded matrix units $e_1, e_2, \dots , e_s$ so that $e = \sum _{i=1}^s \, e_i$ and so there exists $i_0$ so that $(x, y) \in \hat{e}_{i_0}$. If $e_{i_0}$ is $k$-graded then we set $\delta(x,y) \equiv k$. Since summands of finitely graded matrix units maintain their grading from $\psi_z$,  the mapping $\delta $ is well-defined. Furthermore $\delta : \R(\A)\rightarrow \bbZ^+_0$ maintains the same value over the clopen set $\hat{e}_{i_0}$. In other words, if $\delta(x,y)= k$ for some $(x,y) \in \R(\A)$ then the same is true for a neighborhood of $(x,y)$. Hence $\delta $ is continuous.

Assume that $(x, y) , (y, z) \in \R(\A)$. If $(x, y)\in \hat{e}$ and $(y,z) \in \hat{f}$, then $(x,z) \in \widehat{ef}$. Use the claim above to write $e = \sum _{i=1}^s \, e_i$ and $f  = \sum _{j=1}^t \, f_j$ as sums of finitely graded matrix units in some $\A_m$ and so \[
ef = \sum _{i=1}^s  \sum _{j=1}^t \, e_i f_j,
\]
with each $e_if_j$ elementary. If $(x, z) \in \widehat{e_{i_0}f_{j_0}}$, then $(x, y)\in \widehat{e}_{i_0}$ and $(y,z) \in \widehat{f}_{j_0}$. Hence
\begin{align*}
z^{\delta(x,z)}e_{i_0}f_{j_0} &=\psi_z(e_{i_0}f_{j_0}) =\psi_z(e_{i_0})\psi_z(f_{j_0}) \\
&= z^{\delta(x,y)+\delta(y,z)} e_{i_0}f_{j_0},
\end{align*}
and so $\delta $ satisfies the cocycle condition.

We now show that $\delta : \R(\A)\rightarrow \bbZ^+_0$ is coherent. Suppose $(x, y) \in \R(\A)$ with $\delta(x,y) = n$, $n \geq 2$. Hence there exists a finitely graded matrix unit $e$ with $(x, y) \in \hat{e}$ and $\psi_z(e) = z^n e$. By the first paragraph of the proof and the inductivity of $X$, we can $1/2$-approximate $e$ with a sum of the form $\sum _{i=1}^s \, \lambda_i e_i$, $e_i \in X^{ n}\cap \A_m$, $i = 1, 2, \dots , s$, $m \in \bbN$; furthermore, each $e_i$ can be assumed to be a product of matrix units from $X$.  But then all the summands of $e$ in $\A_m$ have to coincide with one of the $e_i$ in the sum above. In particular, there exists $e_i$ which is a product of matrix units from $X$ so that $(x, y) \in \widehat{e_i} $. This shows that $\delta $ is coherent.

Finally assume that $\hdi : \R(\A)\rightarrow \bbZ^+_0$ is another coherent cocycle and let $(x, y) \in \R(\A)$. We claim that $\hdi(x, y) \leq \delta (x,y)$. Indeed if $\hdi(x,y) = n $, then 
\[
(x, y) = (x, x_1)(x_1, x_2)\dots(x_{n-1} , y) 
  \]
 with each one of the factors on the right side of the above equation belonging to $\hdi ^{-1}(\{1\})$; in particular $x_{i-1} \neq x_i$, for all $i=1, 2, \dots ,n$. But then the cocycle condition implies that
\[
\delta (x, y) = \delta (x, x_1)+\delta(x_1, x_2)+\dots + \delta(x_{n-1} , y) \geq n= \hdi(x, y),
  \]
  as desired. By reversing the roles of $\delta$ and $\hdi$ in the above argument we also obtain that $\hdi(x, y) \geq \delta (x,y)$, which shows the uniqueness of $\delta$.
\end{proof}



In order to characterize which limit algebras are isometrically isomorphic to tensor algebras we need two ingredients. One is of course the concept of a $\bbZ^+_0$-valued cocycle that we have been discussing so far. The other one is contained in the following definition. (Compare with \cite[Proposition 4.7]{MSMem}.)

\begin{definition}
Let $\R$ be an AF-semigroupoid on $\D^{\star}$. We say that $\R$ is a \textit{tree} if  for any three points $x, y, z \in \D^{\star}$ with $(x, y), (x, z) \in \R$, we have that $y$ and $z$ are comparable, i.e., either $(z, y) \in \R$ or $(y,z) \in \R$.
\end{definition}

We need the following elementary result.

\begin{lemma} \label{neat}
Let $(X, D)$ be a $\ca$-correspondence with $D$ an abelian $\ca$-algebra. Let $\xi, \eta \in X$ and assume that there exists a positive contraction $d  \in D$ so that $\xi d= \xi$ and $\eta d = 0$. Then, 
\[
\| \xi +\eta\|\leq \max \{\|\xi\|, \|\eta\| \}.
\]
\end{lemma}

\begin{proof}Note that 
\[
\begin{split}
\sca{\xi , \eta}&= \sca{\xi d , \eta} = d^ *\sca{\xi , \eta} \\
                      &= \sca{\xi , \eta}d = \sca{\xi , \eta d} \\
                      & =0.
                     \end{split}
                     \]
Hence,
\[
\begin{split}
0\leq \sca{\xi + \eta , \xi +\eta}&=\sca{\xi , \xi } + \sca{\eta , \eta } =\sca{\xi , \xi d } + \sca{\eta , \eta (1-d) }\\
                                                 &=d^ {1/2}\sca{\xi , \xi }d^ {1/2} + (1-d)^{1/2}\sca{ \eta , \eta}(1-d)^{1/2} \\
                                                 &\leq\|\xi \|^2d + \|\eta\|^2(1-d) \\
                                                 &\leq \max\{\|\xi \|^2, \| \eta \|^2\} 1,
                                                 \end{split}
                                                 \]
as desired.
\end{proof}

We are in position to state and prove the first half of our main result.

\begin{theorem} \label{mainhalf}
Let $\A$ be a triangular limit algebra. If $\A$ is isometrically isomorphic to the tensor algebra of a $\ca$-correspondence, then $\R(\A)$ is a tree admitting a $\bbZ^+_0$-valued continuous and coherent cocycle.
\end{theorem}

\begin{proof}
Let $\big\{ \A_n, \C_n, \D_n, \rho_n\big\}_{n=1}^{\infty}$ be a presentation for $\A$. Hence\break $\D= \varinjlim (\D_n, \rho_n)$ is a canonical masa of the AF $\ca$-algebra $\C= \varinjlim (\C_n, \rho_n)$ and all $\{ \A_n\}_{n=1}^{\infty}$ are digraph subalgebras of $\C_n$ with  $\A_n \cap \A_n^* = \D_n$ and $\rho(\A_n) \subseteq\rho(\A_{n+1})$ for all $n \in \bbN$. Finally, $\A= \varinjlim (\A_n, \rho_n)$.

Assume that there exists a $\ca$-correspondence $(X, D)$ and an isometric isomorphism $\phi: \T^+_X  \rightarrow \A$. Hence $\phi (D) = \D$. Theorem~\ref{main1} implies now the existence of an integer valued continuous and coherent cocycle $\delta : \R(\A) \rightarrow \bbZ^+_0$ with $\delta ^{-1}(\{ 0 \})= \D$. As we saw in the proof of Theorem~\ref{main1}, we also have $\delta ^{-1}(\{1\}) = \widehat{\phi (X)}$.

In order to show that $\R(\A)$ is a tree let $(x, y), (x, z) \in \R(\A)$. We are to show that either $(x, y)$ is a factor of $(x, z)$ or otherwise $(x, z)$ is a factor of $(x, y)$.  If either $x=y$ or $x =z$, then there is nothing to prove.  Otherwise, since $\delta $ is coherent we have 
\[
\begin{split}
(x, y) & = (x, y_1)(y_1, y_2)\dots(y_k , y) \\
 (x, z) &= (x , z_1)(z_1, z_2) \dots (z_l, z)
 \end{split}
 \]
 with each one of the factors on the right side of the above equations belonging to $\delta ^{-1}(\{1\})$. Without loss of generality assume that $k \leq l$.
 
We claim that $y_1=z_1$. Indeed,  by way of contradiction assume otherwise. Since $\delta(x, y_1)=\delta(x, z_1) =1$, it follows from the proof of Theorem~\ref{main1} that there exist matrix units $e, f \in \A_k \cap \phi(X)$, for some $k \in \bbN$, so that $(x, y_1) \in \hat{e}$ and $(x, z_1) \in \hat{f}$. Since $y_1 \neq z_1$, the matrix units $e, f$ have orthogonal initial spaces but the same final space. Hence 
\begin{equation} \label{sqrt2}
\|e+f\|= \sqrt{2}.
\end{equation}
On the other hand let $e = \phi(\xi)$ and $f = \phi(\eta)$ for some $\xi , \eta \in \D^{\star}$. Then Lemma~\ref{neat}, with $d  = \phi^{-1} ( e^*e)$, implies that 
\[
 \|e + f\| = \|\xi + \eta \|  \leq \max \{\|\xi\|, \|\eta\|\} = \max\{\|e\|, \|f\|\}= 1.
\]
This contradicts (\ref{sqrt2}) and shows that $y_1=z_1$.

Continuing in that fashion we can show now that $y_2 = z_2$, $y_3 = z_3$ and so on. This implies that $(x, y)$ is a factor of $(x, z)$ and so $\R(\A)$ is a tree.
\end{proof}

We now focus on the converse of Theorem~\ref{mainhalf}. For the rest of the section we assume that $\A$ is a triangular limit algebra so that $\R(\A)$ is a tree admitting a $\bbZ^+_0$-valued continuous and coherent cocycle $\delta : \R(\A) \rightarrow \bbZ^+_0$. In what follows $\big\{ \A_n, \C_n, \D_n, \rho_n\big\}_{n=1}^{\infty}$ will always denote a presentation for $\A$. 

We say that a matrix unit in $e \in \A$ is $1$-elementary iff $\hat{e} \subseteq \delta^ {-1}(\{1\})$. A matrix unit is said to be $k$-elementary, for some $k \geq2$, if it is the product of $k$ $1$-elementary matrix units. A matrix unit will be called elementary iff it is $k$-elementary for some $k \in \bbN$. (Compare these definitions with the second paragraph of the proof of Theorem~\ref{main1}.) The collection of all $k$-elementary matrix units is denoted as $X_k$, $ k \in \bbN$. 

\begin{lemma} \label{rightsum}
Let $\A=  \varinjlim (\A_n, \rho_n)$ be a triangular limit algebra so that $\R(\A)$ is a tree admitting a $\bbZ^+_0$-valued continuous and coherent cocycle $\delta : \R(\A) \rightarrow \bbZ^+_0$. Then any non-diagonal matrix unit in $\A$ can be written as the sum of elementary matrix units from some $\A_n$, $n \in \bbN$.
\end{lemma}

\begin{proof}
It suffices to show that the graphs of all elementary matrix units form a basis for the topology of $\R(\A)$. Hence given a matrix unit $e \in \A$ and $(x, y) \in \hat{e}$, we need to produce an elementary matrix unit $g \in \A$ so that $(x, y) \in \hat{g} \subseteq \hat{e}$.

Towards this end, assume that $\delta(x, y) =k$. 
Since $\delta $ is coherent, there exist $x_1, x_2, \dots x_{k-1} \in \D^{\star}$ so that $(x_i, x_{i+1}) \in \R(\A)$ and $\delta(x_i,x_{i+1})=1$, for all $ i = 0, 1, \dots k-1$. (Here we understand that $x_0=x$ and $x_k =y$.) The continuity of $\delta $ now implies the existence of matrix units $f_1, f_2, \dots f_k$ in $\A$ so that $(x_i, x_{i+1}) \in \hat{f}_{i+1}$ for all $ i = 0, 1, \dots k-1$. Consider an $n \in \bbN$ so that $e$ and $f_1, f_2, \dots f_k$ all belong to some $\A_n$. If $c \in \D_n$ is a diagonal matrix unit with $(y,y) \in \hat{c}$, then $ec$ is a matrix unit in $\A_n$ which is a summand of $e$ and so $\widehat{ec} \subseteq \hat{e}$. On the other hand, there exist matrix units $g_1, g_2, \dots g_k$ in $\A_n$ so that $f_1f_2\dots f_k c = g_1g_2\dots g_k$; since $\hat{g_i}\subseteq \hat{f_i}$ for all $ i = 0, 1, \dots k-1$, the matrix units $g_i$ are $1$-elementary. Hence,  $g = g_1g_2\dots g_k$ is k-elementary. Furthermore, the matrix units $ec$ and $g$ of $\A_n$ have the same final projection (since $(x, x)$ belongs to both $\widehat{ec}$ and $\hat{g}$) and so they agree, i.e., $g =ec$. Conclusion: $(x, y) \in \hat{g}=\widehat{ec}\subseteq \hat{e}$ and $g$ is elementary, as desired. 
\end{proof}

In the next lemma we will be using the following fact. Let $\{e_{i,j}\}$ be a matrix unit system for some finite dimensional $\ca$-algebra and let $\S \subseteq \{e_{i,j}\}$ be some subset with the property that if $e_{i,j} \in \S$ then $e_{ii}, e_{jj} \in \S$. We denote by $G(\S)$ the directed graph whose nodes are all matrix units of the form $e_{ii} \in \S$ and the edges are determined as follows: there exists an edge from $e_{jj}$ to $e_{ii}$ if and only if $e_{ij} \in \S$. If $G(\S)$ happens to be an out-forest, then the algebra $\alg(\S)$ generated by $\S$ is completely isometrically isomorphic to the tensor algebra $\T^+_{G(\S)}$. This fact can be seen in many ways. For instance, one can verify that $\alg(\S)$ is a tree algebra in the sense of \cite[Definition 2.2]{DPP} and then appeal to \cite[Theorem 4.1]{DPP}. (See also \cite{Dav}.)

\begin{lemma} \label{Bn}
Let $\A=  \varinjlim (\A_n, \rho_n)$ be a triangular limit algebra so that $\R(\A)$ is a tree admitting a $\bbZ^+_0$-valued continuous and coherent cocycle $\delta : \R(\A) \rightarrow \bbZ^+_0$. Then there exists an ascending sequence $\{\B_n\}_{n=1}^{\infty}$ of finite dimensional subalgebras of $\A$ so that

\begin{itemize}
\item[(i)] $\bigcup_{n=1}^{\infty} \B_n= \bigcup_{n=1}^{\infty} \A_n$, and,
\item[(ii)] each $\B_n$ is completely isometrically isomorphic to the tensor algebra of an out-forest.
\end{itemize} 
\end{lemma}

\begin{proof}
For each $ n \in \bbN$, let $\B_{n, 1}$ be the collection of all $1$-elementary matrix units in $\A_n$ and let $\B_n$ be the finite dimensional subalgebra of $\A_n$ generated by $\B_{n, 1}$ and $\D_n$. Any $1$-elementary matrix unit $e \in A_n$ can expressed as a sum $ \sum e_j$ of matrix units in $ e_j \in \A_{n+1}$;  since $\hat{e_j} \subseteq \hat{e}$, all these $e_j$ are also $1$-elementary. This shows that the sequence $\{\B_n\}_{n=1}^{\infty}$ is ascending.

By Lemma~\ref{rightsum}, any matrix unit in $\A$ is either diagonal and so it belongs to some $\D_n$, $n \in \bbN$, or otherwise, it is a sum of elementary units in some $\A_n$, $n \in \bbN$, and so it belongs to the algebra generated by $\B_{n, 1}$. This shows that $\cup_{n=1}^{\infty} \B_n$ contains all matrix units in $\A$ and so $\cup_{n=1}^{\infty} \B_n= \cup_{n=1}^{\infty} \A_n$.

It remains to show that each $\B_n$ is completely isometrically isomorphic to the tensor algebra of a tree. Fix an $n \in \bbN$ and let $G_n$ be the finite graph with nodes the matrix units in $\D_n$ and (directed) edges the elements of $\B_{n, 1}$, i.e., $G_n = G(\D_n \cup \B_{n, 1})$. Clearly $\B_n =\alg(G_n)$. If we show that $G_n$ is a directed tree where no vertex receives more than one edge, then the previous discussion will imply that $\B_n \simeq \T^+_{G_n}$.

First notice that no vertex of $G_n$ receives more than one edge. Indeed, assume to the contrary that there exists node $c \in D_n$ which receives two distinct edges $e_1$ and $e_2$. This implies that there exist \textit{distinct} $x, y, z \in \D^{\star}$ so that $(x, y) \in \hat{e_1}$ and $(x, z) \in \hat{e_2}$. Since $\R(\A)$ is a tree either $(y, z)$ or $(z, y)$ belongs to $\R(\A)$, say $(y, z) \in \R(\A)$. But then the cocycle condition implies 
\[
1= \delta(x,z)=\delta(x,y)+\delta(y,z)= 1+\delta(y, z)
\]
and so $\delta(y, z)=0$. Since $\delta ^{-1}(\{0\}) = \widehat{\D}$, we obtain $y =z$, which is a contradiction. 

Having established that no vertex of $G_n$ receives more than one edge, we can now see that $G_n$ has no (undirected) cycles. Indeed if such a cycle existed then it would also have to be directed cycle or otherwise  a vertex of $G_n$ would receive two edges. But then the presence of such a directed cycle would imply that $\A$ contains both a matrix unit and its adjoint. 
Which in turn implies that there exist distinct $x , y \in \D^{\star}$ so that $(x,y), (y,x)\in \R(\A)$. But then, by the cocycle condition
\[
0= \delta(x, x)=\delta(x,y)+\delta(y,x)\geq2
\]
which is a contradiction.
Therefore, the proof is established.
\end{proof}

For the proof of the main theorem below, we also need a result from \cite{KR} which we have labeled there as the Extension Theorem. Below we just state a special case of it. 

\begin{lemma} \label{extthm}
Let $\D \subseteq \C$ be an inclusion of $\ca$-algebras and let $X \subseteq  \C$ be a closed $\D$-bimodule with $X^*X\subseteq \D$. If $\A = \overline{\alg}(X\cup \D)$ and $s \in l^2 (\bbN)$ denotes the forward shift, then the following are equivalent
\begin{itemize}
\item[(i)] $\A$ is completely isometrically isomorphic to the tensor algebra \break $\T_{(X, \D ) }^+$ via a map that sends generators to generators.
\item[(ii)] The association
\begin{equation} \label{extensioncriterion}
\begin{aligned}
\D \ni d &\longrightarrow d \otimes 1, \\
X \ni  \xi &\longrightarrow \xi \otimes s
\end{aligned}
\end{equation}
extends to a well defined completely isometric map on $\alg (X \cup \D)$.
\end{itemize}
\end{lemma}

\begin{theorem} \label{main}
Let $\A$ be a triangular limit algebra. Then $\A$ is isometrically isomorphic to the tensor algebra of a $\ca$-correspondence if and only if $\R(\A)$ is a tree admitting a (necessarily unique) $\bbZ^+_0$-valued continuous and coherent cocycle.
\end{theorem}

\begin{proof}
One direction of the theorem follows from Theorem~\ref{mainhalf}. In order to verify the other direction, assume that  $\R(\A)$ is a tree admitting a $\bbZ^+_0$-valued continuous and coherent cocycle.

Let $X = \overline{[ X_1]}$, where $X_1$ denotes as before the set of all $1$-elementary matrix units. Clearly $X\subseteq \C$ is a $\D$ bimodule. We claim that $X^*X \subseteq \D$. Indeed if $e, f \in X$ are matrix units then there exist matrix units $e_1, f_1 \in X_1\cap\A_n$, for some $n \in \bbN$, so that $e^*f=e^*_1f_1$. However the matrix units in  $X_1\cap\A_n = \B_{n,1}$ form the edges of the graph $G_n$ where no vertex receives more than one edge. (See the proof of Lemma~\ref{Bn}.) Hence either 
$e^*_1f_1 = 0$ or otherwise $e_1=f_1$ and so 
\[
e^*f=e^*_1f_1= e^*_1e_1 \in \D.
\]
This suffices to show that $X^*X \subseteq \D$.

By Lemma~\ref{Bn} we also have that $\overline{\alg}(X, \D) = \A$. Hence Lemma~\ref{extthm} will imply that $\A = \T^+_{(X, \D)}$ as soon as we verify (\ref{extensioncriterion}). This is done as follows.
Since each $\B_n$ is the tensor algebra of a graph, Lemma~\ref{extthm} shows that the association
\begin{equation*} 
\begin{aligned}
\D_n \ni d &\longrightarrow d \otimes 1, \\
\B_{n,1} \ni  e &\longrightarrow e\otimes s
\end{aligned}
\end{equation*}
extends to a completely isometric map $\phi_n$ defined on $\B_n$. Furthermore since each element in $\B_{n,1}$ can be written as a sum of elements in $\B_{n+1, 1}$, we have ${\phi_{n+1}} \mid_{\B_n}=\phi_n$, for all $n \in \bbN$. Hence we obtain a completely isometric map $\phi$ defined on 
\[
\overline{\cup_{n=1}^{\infty} \B_n}= \overline{\cup_{n=1}^{\infty} \A_n} = \A
\]
and satisfying (\ref{extensioncriterion}), as desired.
\end{proof}

\begin{remark}
Theorem \ref{main} also clarifies another issue in the literature. In \cite{DonH} Donsig and Hopenwasser study non-selfadjoint (analytic) subalgebras of $\ca$-crossed products by partial actions. Their theory incudes as examples many familiar non-selfadjoint operator algebras, including the standard and refinement limit algebras. Even though the $\ca$-crossed products by partial actions can be described as Cuntz-Pimsner  algebras of $\ca$-correspondences, Theorem  \ref{main} says that their analytic subalgebras may fail to be tensor algebras. This is a new phenomenon due to the partial action, since the analytic subalgebras of $\ca$-crossed products by automorphisms are always tensor algebras.
\end{remark}


\section{Special cases and applications}
We address now some of the consequences of Theorem~\ref{main}. In \cite{DKDoc} Davidson and Katsoulis raised the question whether the class of semi-Dirichlet algebras coincides with the class of tensor algebras of $\ca$- correspondences. The question was actually open even for the subclass of Dirichlet algebras and it was settled in the negative by Kakariadis \cite{Kak} who produced various examples of Dirichlet function algebras which are not \textit{completely} isometrically isomorphic to tensor algebras. We now use Theorem~\ref{main} to produce additional examples of Dirichlet algebras, which actually fail to be isomorphic to tensor algebras even by isometric isomorphisms. The second author also echoed these questions in \cite{Ram}.

First, let us recall the important definitions. An operator algebra $\A$ is a \textit{Dirichlet algebra} if $\A + \A^*$ is dense in $\cenv(\A)$ and we say that $\A$ is a \textit{semi-Dirichlet algebra} if $\A^*\A \subset \overline{\A + \A^*}$ when considered as a subalgebra of $\cenv(\A)$. This language is logical as it was proven in \cite{DKDoc} that $\A$ is Dirichlet if and only if $\A$ and $\A^*$ are semi-Dirichlet.

 Recall as well that a (regular) limit algebra $\A = \varinjlim ( \A_n, \rho_n)$ is said to be a \textit{full nest algebra}, if each $\A_n$ is isomorphic to the $k_n \times k_n$ upper triangular matrices $\T_{k_n}$ and the (regular, $*$-extendable) embeddings $\rho_n$ satisfy $\rho_n (\lat \A_n) \subseteq \lat (\A_{n+1})$ (such embeddings are called \textit{nest embeddings}). The prototypical example of a full nest algebra is the \textit{refinement limit algebra} (see Example~\ref{tra}) but there are many more. They were first studied in \cite{HP}.

\begin{proposition} \label{counterexample}
A full nest algebra $\A$ is a Dirichlet algebra which is not isometrically isomorphic to the tensor algebra of any $\ca$-correspondence.
\end{proposition}

\begin{proof} Any non-diagonal matrix unit in a full nest algebra $\A$ can be written as a sum of matrix units where at least one of them is the product of two distinct non-diagonal matrix units. Hence $\A$ cannot support any integer valued cocycle $\delta: \R(\A)\rightarrow \bbZ^+_0$ satisfying $\delta(\hat{e}) \subseteq \{1\}$ for some non-diagonal matrix unit $e \in \A$. However as we saw in the proof of Theorem~\ref{main1} such a cocycle and matrix units do exist for an $\A$ isometrically isomorphic to a tensor algebra.
\end{proof}

Of course there are many more counterexamples: the alternation algebras, the algebras $\A(\bbQ)$ \cite{Powlex} and many more. All these counterexamples contradict the intuition coming from our finite dimensional experience, where all maximal triangular algebras are actually tensor algebras.

Next we want to elaborate further on the refinement embedding and produce additional examples of limit algebras with interesting properties. For our next class of examples, the graphs that we will be considering are out-trees. This allows us to move freely between the two concepts. Note that these are the only graphs with satisfy such a property and at the same time their $\ca$-envelope is a full matrix algebra.

\begin{example}[Tree-refinement algebras] \label{tra}
Let $G$ be any out-tree and let $\{k_n\}$ be a sequence of positive integers with $k_1=|G^0|$.  The \textit{tree refinement algebra} $\A\big(G, \{k_n\}\big)=\varinjlim (\T_{G_n}^+, \sigma_n)$ is a direct limit of graph algebras, which we now describe. 

Let $\{ e_{ij}^n\}$ be a matrix unit system for $ M_{k_n}$ and let $\sigma_n: M_{k_n}\rightarrow M_{k_{n+1}}$ be the refinement embedding, $n \in \bbN$. For notational simplicity we write 
\begin{equation} \label{esum}
\sigma_n(e_{i,j}^n) = \sum_s \, e_{(i, s) (j,s)}^{n+1}\in M_{k_{n+1}}(\bbC),
\end{equation}
i.e., $e_{(i, s) (j,s)}^{n+1} = e_{(i-1)l_n +s,  (j-1)l_n +s}^{n+1}$, where $l_n=k_{n+1}/k_n$.

Let $G_1=G$ and assume that each one of the edges of $G_1$ corresponds uniquely to a matrix unit $e_{ij}^1$ with $i\neq j$, i.e., $G_n^1\subseteq \{e_{i,j}^1\}_{i\neq j}$, and similarly, $G_n^0\subseteq \{e_{ii}\}$, so that these matrix units allow us to represent $\T^+_{G_1}$ faithfully inside $M_{k_1}(\bbC)$. We construct now inductively out-trees $G_2, G_3 \dots $, which we call \textit{ampliations} of $G$, under the following scheme. 

Assume that $G_n$ has been constructed so that each of the edges of $G_n$ corresponds uniquely to a matrix unit $e_{ij}^n$ with $i\neq j$, i.e., $G_n^1\subseteq \{e_{i,j}^+\}_{i\neq j}$, and similarly, $G_n^0\subseteq \{e_{ii}^n\}$. Once again these matrix units allow us to consider $\T^+_{G_n}$ inside $M_{k_n}(\bbC)$. We consider now the graph $G_{n+1}$ so that the diagonal matrix units of $M_{k_{n+1}}(\bbC)$ will comprise the vertex set $G_{n+1}^0$ and the edge set $G_{n+1}^{1}$ consists of the following matrix units 
\begin{itemize}
\item $e_{(i, s+1) (i,s))}^{n+1}$, $1\leq i \leq k_n$ and $1 \leq s\leq k-1$
\item $e_{(i, 1) (j,l_n)}^{n+1}$, provided that $e_{ij}^n \in \T_{G_n}^+$, 
\end{itemize}
where $l_n=k_{n+1}/k_n$. Pictorially, the vertices $j,i$ and the edge $e_{i,j}^n$ of $G_n$ are ``being replaced" in $G_{n+1}$ by the vertices and edges appearing below 

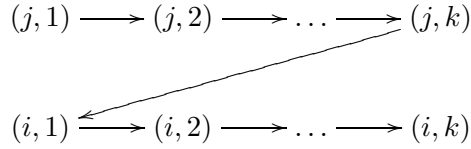
\vspace{.05in}

\[
\xymatrix{{(j,1)} \ar[r] & {(j,2)} \ar[r] &\dots \ar[r] &{(j,k)} \ar[dlll]  \\
{(i,1)} \ar[r]  &{(i,2)} \ar[r] & \dots \ar[r] &{(i,k)} }
\]

\vspace{.1in}

\noindent For instance, if $G_n$ is the $\Lambda$-graph

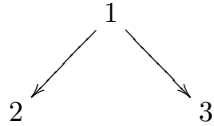
\vspace{.05in}

\[
\xymatrix{& 1 \ar[dl] \ar[dr]& \\ 
2 & &3}
\]

\vspace{.05in}

\noindent and $\sigma_n$ is the refinement embedding of multiplicity $2$, then $G_{n+1}$ will be the graph

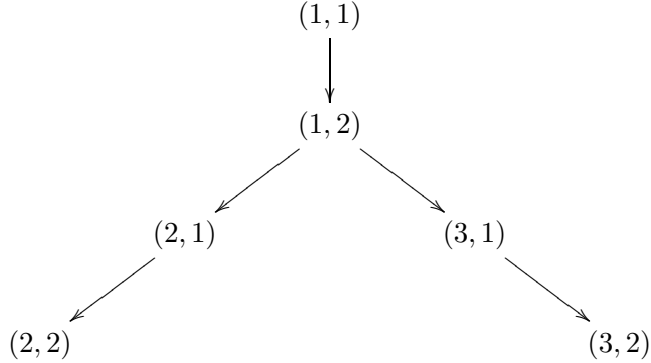
\vspace{.1in}

\[
\xymatrix{&&(1,1) \ar[d]&&\\
&&(1,2) \ar[dl] \ar[dr]&&\\
&(2,1) \ar[dl]&&(3,1) \ar[dr]&\\
(2,2)& &&&(3,2)}
\]

\vspace{.05in}

It is easy to see that $\sigma_n(\T_{G_n}^+)\subseteq \T_{G_{n+1}}^+$. Indeed, if $\sigma(e_{i,j}^n )$ is written as in (\ref{esum}), then 
\[
e_{(i, s) (j,s)}^{n+1} =  \big( \prod_{z=1}^{s-1} \, e_{(i,z+1) (i,z)}^{n+1}\big) e_{(i, 1)(j,k)}^{n+1} \big( \prod_{w=s}^{k-1} \, e_{(j,w+1) (j,w)}^{n+1} \big) \in  \T_{G_{n+1}}^+.
\]
Therefore we have a directed limit $\A=  \varinjlim (\T_{G_n}^+, \sigma_n)$; the resulting limit algebra is the tree-refinement algebra $\A\big(G, \{k_n\}\big)$.
\end{example}

In the case where $G_1$ is a total ordering, the resulting tree-refinement algebra coincides with one of the familiar refinement limit algebras. The situation is more interesting however when $G_1$ is not a total ordering. The proof of the next result is similar to that of Proposition~ref{counterexample}.

\begin{proposition} \label{secondex}
If $G$ is not a total ordering then the tree-refinement algebra $\A\big(G, \{k_n\}\big)$ is a semi-Dirichlet algebra which is neither a Dirichlet algebra nor isometrically isomorphic to a tensor algebra.
\end{proposition}

 An interesting feature of the limit algebras which also happen to be tensor algebras, comes from the proof of Lemma~\ref{Bn}.
 
\begin{definition} 
Let $G=\cup_{k=1}^l G_k$ be an out-forest which is the  disjoint union of out-trees $G_1, G_2, \dots, G_l$ and let $H$ be an out-tree. We say that a $*$-extendable embedding of the form
\[
\rho:\T^+_{G} =\oplus_{k=1}^l \T^+_{G_k}\longrightarrow \T^+_H
\]
is a \textit{tree-standard embedding} iff for any $f \in G^1$ we have $\rho(L_e) = \sum_{f } L_f$, for suitable $f \in H^1$. By a slight abuse of terminology, direct sums of tree-standard embeddings will also be called tree-standard.
\end{definition}

It is easy to see that a tree-standard embedding $\rho: \oplus_{k=1}^l \T^+_{G_k}\longrightarrow \T^+_H$, where $G_1, G_2, \dots, G_n, H$ are out-trees, maps the trees $G_k$ onto ``branches" of $H$ by ``joining" the root of each $G_k$ with one of the vertices of $H$ by a single vertex. Furthermore,  a direct sum of such embeddings produces the generic tree-standard embedding. Clearly all tree-standard embeddings are regular.

The next result gives a characterization, in terms of a presentation, for limit algebras which also happen to be isometrically isomorphic to tensor algebras.

\begin{theorem} \label{presentation1}
Let $\A = \varinjlim (\A_n, \rho_n)$ be a triangular limit algebra. Then $\A$ is isometrically isomorphic to the tensor algebra of a $\ca$-correspondence if and only if $\A$ admits a (perhaps different) presentation $\A = \varinjlim (\tilde{\A}_n, \tilde{\rho}_n)$, where each $\tilde{\A}_n$ is the tensor algebra of an out-forest and the embeddings $\tilde{\rho}_n$ are tree-standard embeddings.
\end{theorem}

\begin{proof} The algebras $\tilde{\A}_n$ are exactly the algebras $\B_n$ appearing in the proof of 
 Lemma~\ref{Bn} and the tree-standard embeddings $\tilde{\rho}_n$ are precisely the ones arising from the inclusions $\B_n \subseteq \B_{n+1}$, $ n \in \bbN$.
\end{proof}
The previous result will allow us now to conclude that the (unique) cocycle appearing in Theorem~\ref{main} has a very specific description.

Let $\A = \varinjlim (\A_n, \rho_n)$, where each $\A_n$ is the tensor algebra of an out-forest $G_n$ and the embeddings $\rho_n$ are tree-standard embeddings. Given $(x, y ) \in \R(\A)$ we define
\begin{equation} \label{countingcoc}
\hdi(x,y)= k, \mbox{ where }(x, y) \in \widehat{L}_u , \mbox{ with } u \in G_n^k, \, \, n, k \in \bbN.
\end{equation}
Because the embeddings $\rho_n$ are tree-standard, the quantity $\hdi(x,y)$ is well defined, i.e., it is independent of the $n \in \bbN$ appearing in (\ref{countingcoc}). Furthermore, as a cocycle, $\hdi$ is coherent and continuous. We call $\hdi: \R(\A)\rightarrow \bbZ_0^+$ the \textit{counting cocycle}. The name comes from the strongly maximal TAF literature \cite{PPWcoc}. Indeed let $\A = \varinjlim (\A_n, \rho_n)$ be a strongly maximal  TAF $\ca$-algebra and let $\R(\A)$ be the associated groupoid. If $(x, y) \in \R(\A)$ then we define
\begin{equation} \label{countingcocycle}
\hat{\delta}(x,y) = \sup \{ j-i\mid (x, y) \in \widehat{e}_{i, j}, \mbox{ with }e_{i,j} \in \A_m, \mbox{ for some } m \in \bbN \}.
\end{equation}
This supremum is not always finite but whenever it is for all $(x, y) \in \R(\A)$, then it defines the \textit{counting cocycle} $\delta  : \R (\A) \rightarrow \bbZ$. It is a consequence of our theory that (\ref{countingcoc}) and (\ref{countingcocycle}) are compatible for the semigroupoid of a triangular limit algebra which is also isomorphic to a tensor algebra.

\begin{corollary}
Let $\A$ be a triangular limit algebra. If $\A$ is isometrically isomorphic to the tensor algebra of a $\ca$-correspondence, then the counting cocycle is the unique $\bbZ^+_0$-valued coherent cocycle on $\R(\A)$.
\end{corollary}

Theorem~\ref{presentation1} also gives an alternate proof of an old result of Poon and Wagner \cite{PW}.

\begin{proposition}\cite[Theorem 2.9]{PW} \label{presentation}
Let $\A = \varinjlim (\A_n, \rho_n)$ be a strongly maximal TAF algebra. Then $\A$ is $\bbZ$-analytic if and only if $\A$ admits a (perhaps different) presentation $\A = \varinjlim (\tilde{\A}_n, \tilde{\rho}_n)$, where each $\tilde{\A}_n$ is a direct sum of upper triangular matrix algebras and the regular $*$-extendable embeddings $\tilde{\rho}_n$ are direct sums of standard embeddings.
\end{proposition}

\begin{proof} The spectrum $\R(\A)$ of a strongly maximal TAF algebra is always a tree. Hence the presence of an integer valued cocycle is equivalent to being isomorphic to a tensor algebra. Now the upper triangular matrices are isomorphic to the tensor algebras of linear orderings and so the algebras $\B_n$ appearing in the proof of Lemma~\ref{Bn} are just direct sums of such algebras.
\end{proof}

The following example illustrates the various possibilities that arise in the presentation of a strongly maximal TAF which is also happens to be a tensor algebra. In particular the change of presentation hinted in Theorem~\ref{presentation1} and Proposition~\ref{presentation} may be unavoidable.

\begin{example}
Let $\A = \varinjlim (M_{3^n}, \rho_n)$ be the TUHF algebra with $\rho_n$ defined as
{\small \[
\left[\begin{array}{ccc} A_{11} & A_{12} & A_{13} \\ & A_{22} & A_{23} \\ &&A_{33} \end{array}\right]
\mapsto \left[\begin{array}{ccccccccc} A_{11} & A_{12} &&&A_{13} \\ & A_{22} &&& A_{23} \\ &&A_{11} & A_{12} && A_{13} \\ &&& A_{22} && A_{23} \\ &&&&A_{33} \\ &&&&&A_{33} \\ &&&&&& A_{11} & A_{12} & A_{13} \\
&&&&&&& A_{22} & A_{23} \\ &&&&&&&&A_{33} \end{array}\right]
\]}
where $A_{ij} \in M_{3^{n-1}}$. Another view of this is the block map
\[
\left[\begin{array}{cc} A & B \\  & C \end{array}\right] \mapsto
\left[\begin{array}{cccccc} A && B \\ & A && B \\ && C \\ &&&C \\ &&&&A & B \\ &&&&&C \end{array}\right]
\]
where $A\in M_{2\cdot 3^{n-1}}, C \in M_{3^{n-1}}, B\in M_{2\cdot 3^{n-1}, 3^{n-1}}$.

 Let $e \in M_{3^n}$ be any matrix unit with $e = e_{i_1, j_1}^{n+1} + e_{i_2, j_2}^{n+1} + e_{i_3, j_3}^{n+1}$, $e_{i_s, j_s}^{n+1} \in  M_{3^{n+1}}$. If $(x, y) \in \hat{e}$ then there exists $s=1,2,3$ so that $(x, y) \in \hat{e}_{i_s, j_s}^{n+1}$. In that case it is easy to see that
\begin{equation*}
\sup \{ j-i\mid (x, y) \in \widehat{e}_{i, j}, \mbox{ with }e_{i,j} \in M_{3^m}, \mbox{ for some } m \in \bbN \}= j_s - i_s.
\end{equation*}
Hence the supremum in (\ref{countingcocycle}) exists and therefore the counting cocycle $\hat{\delta}$ is well defined on the semigroupoid of $\A$. Furthermore $\hat{\delta}$ maintains the same value in a neighborhood of $(x, y)$, namely $\hat{e}_{i_s, j_s}^{n+1}$. Hence $\hat{\delta}$ is continuous and therefore Theorem~\ref{main} shows that $\A$ is the tensor algebra of a $\ca$-correspondence. Note however that the presentation  $\A = \varinjlim (T_{3^n}, \rho_n)$ by no means involves direct sums of standard embeddings. Nevertheless Theorem~\ref{presentation1} guarantees that such a presentation does exist.
\end{example}

\section{classification}

In this section we will show that if $\A$ is a triangular limit algebra which also happens to be a tensor algebra, then $\A$ is actually the tensor algebra of a topological graph which is naturally associated with $\R(\A)$. We will subsequently show that these graphs form a complete isomorphism invariant for these algebras.

 A topological graph $G= (G^{0}, G^{1}, r , s)$ consists of two locally compact spaces $G^0$, $G^1$, a continuous map $r: G^1 \rightarrow G^0$ and a local homeomorphism $s: G^1 \rightarrow G^0$. The set $G^0$ is called the base (vertex) space and $G^1$ the edge space. When $G^0$ and $G^1$ are both equipped with the discrete topology, we have a discrete countable graph (see below).

With a topological graph $G=  (G^{0}, G^{1}, r , s)$ there is a $\ca$-correspondence $X_{G}$ over $C_0 (G^0)$. The right and left actions of $C_0(G^0 )$ on $C_c ( G^1)$ are given by
\[
(fFg)(e)= f(r(e))F(e)g(s(e))
\]
for $F\in C_c (G^1)$, $f, g \in C_0 (G^0)$ and $e \in G^1$. The inner product is defined for $F_1, F_2 \in C_c ( G^1)$ by
\[
\left< F_1 \, | \, F_2\right>(v)= \sum_{e \in s^{-1} (v)} \overline{F_1(e)}F_2(e)
\]
for $v \in G^0$. Finally, $X_{G}$ denotes the completion of $C_c ( G^1)$ with respect to the norm
\begin{equation} \label{norm}
\|F\| = \sup_{v \in G^0} \left< F \, | \, F\right>(v) ^{1/2}.
\end{equation}

When $G^0$ and $G^1$ are both equipped with the discrete topology, then the tensor algebra $\T_{G}^+ \equiv \T^+_{X_{G}}$ associated with $G$ coincides with the quiver algebra of Muhly and Solel \cite{MS}. 

Now let $\A$ be a triangular limit algebra with presentation\break $\big\{ \A_n, \C_n, \D_n, \rho_n\big\}_{n=1}^{\infty}$ and assume that $\A$ admits a continuous and coherent cocycle $\delta : \R(\A) \rightarrow \bbZ^+_0$. Let $G^0 =\D^{\star}$ be the Gelfand space of $\A \cap\A^*$ and let 
\begin{equation} \label{G1}
G^1=\{ (x, y) \in \R(\A)\mid  \delta(x, y) =1\}. 
\end{equation} 
Given $(x, y) \in G^1 \subseteq \R(\A)$, let $r(x, y)=x$ and $s(x,y)=y$. The quadruple $G=  (G^{0}, G^{1}, r , s)$ $\A$ is said to be the \textit{graph associated with  $( \R(\A), \delta)$.}
 
Assume further that $\A$ is isometrically isomorphic to the tensor algebra of a $\ca$-correspondence.
By Theorem~\ref{presentation} we may alter the presentation for $\A$ so that each $\A_n\simeq \T_{G_n}$ is the tensor algebra of an out-forest $G_n$ and the embeddings $\rho_n$ are tree-standard embeddings. Let $X_1$ be the collection of all matrix units in $\A$ corresponding to the edges of all $G_n$, $n=1, 2, \dots$, and let $X = \overline{[ X_1]}$. As we saw in Theorem~\ref{main}, the pair $(X, \D)$ is a $\ca$-correspondence and $\A\simeq \T_X^+$. Furthermore
\[
G^1=\bigcup_{e \in X_1} \hat{e}.
\]
The collection $\{\hat{e} \mid e \in X_1\}$ forms a base for a topology on $G^1$ and $G^1$ equipped with that topology becomes a locally compact Hausdorff space. 

\begin{lemma} \label{approx}
 If $G=  (G^{0}, G^{1}, r , s)$ is as above, then $G$ is a topological graph. Furthermore, any element of $X_G$ can be approximated by finite sums of the form $\sum_{i=1}^n \lambda_i \chi_{\hat{e}_i}$, where $\{\lambda_i\}_{i=1}^n$ are scalars, $\{ e_i\}_{i=1}^n$ are matrix units with disjoint graphs and $\chi_{\hat{e}}$ denotes the characteristic function of $\hat{e}$. 
\end{lemma}

\begin{proof} If $e \in \A$ is any matrix unit then $s_{\mid \hat{e}}$ becomes a homeomorphism onto a clopen subset of $\D^{\star}$. Thus $s: G^1\rightarrow G^0$ is a local homeomorphism and so $G$ is a topological graph. 

Now notice that if $e , f \in \A$ are matrix units then either $\hat{e} \cap \hat{f}=\emptyset$ or otherwise $\hat{e} \subseteq \hat{f}$ or  $\hat{f} \subseteq \hat{e}$. Furthermore in the case of a containment, say $\hat{e} \subseteq \hat{f}$, it is easy to see that there exists matrix units $e_1, e_2, \dots e_m$ with disjoint graphs so that $\hat{f} \backslash \hat{e} = \cup_{i=1}^m\hat{e}_i$. This implies the following: if $f_1, f_2, \dots ,f_k$ are matrix units in $\A$, then there exist matrix units $e_1, e_2, \dots e_n$ with disjoint graphs so that $ \cup_{j=1}^k\hat{f}_j =  \cup_{i=1}^n\hat{e}_i$ and each $\hat{e_i}$ is contained in some $\hat{f_j}$. The second statement of the lemma follows easily from that fact.

Indeed it is enough to verify the approximation claim for any $F \in C_c(G^1)$. Let $\epsilon \geq 0$. By compactness there are finitely many matrix units $f_1, f_2, \dots ,f_k$ and scalars $\mu_1, \mu_2, \dots \mu_k$ so that for each $x \in \supp f$ there exists $j = 1, 2, \dots k$ so that $x \in \hat{f}_j $ and $| F(x) -\mu_j |\leq \epsilon$. Refine the partition $\{ \hat{f}_1, \hat{f}_2, \dots , \hat{f}_k\} $ into a partition $\{ \hat{e}_1, \hat{e}_2, \dots , \hat{e}_n\} $ using the previous paragraph and the conclusion follows. 
\end{proof}

\begin{theorem} \label{maindescr}
Let $\A$ be a triangular limit algebra which is isometrically isomorphic to the tensor algebra of a $\ca$-correspondence and let $\delta: \R(\A) \rightarrow \bbZ^+_0$ be the unique continuous coherent cocycle on the semigroupoid of $\A$. If $G=  (G^{0}, G^{1}, r , s)$ is the topological graph associated with  $( \R(\A), \delta)$, then $\A$ is isometrically isomorphic to  $\T^+_G$
\end{theorem}

\begin{proof} Consider the following covariant representation $(\pi, t)$ of the $\ca$-correspondence $\big(X_G, C(G^0)\big)$. Let $\F : \D \rightarrow C(G^0)$ denote the Gelfand transform and let $s$ be the forward shift acting on $l^2(\bbN)$. Then let 
\[
\pi: C(G^0)\longrightarrow \D\otimes I : f \longmapsto \F^{-1}(f) \otimes I.
\]
The mapping $t : X_G\rightarrow X\otimes\ca(s)$ is defined as follows: if $\xi = \sum_{i=1}^n \lambda_i \chi_{\hat{e}_i}$ is as in Lemma~\ref{approx}, then 
\[
t(\xi)=\big(\sum_{i=1}^n \lambda_i e_i\big)\otimes s.
\]
It is easy to see that $t(\xi)$ is well-defined and $\|t(\xi)\|= \|\xi\|$. Lemma~\ref{approx} shows now that $t$ extends to an isometry $t : X_G\rightarrow X_1\otimes\ca(s)$ and that the pair $(\pi, t)$ is a covariant representation of  $\big(X_G, C(G^0)\big)$.

One observes now that the representation $(\pi , t)$ satisfies the requirements of Katsura's gauge invariance uniqueness theorem and therefore $(\pi\times t)(\T^+_{X_G}) \break \simeq \T_{X_G}$. However, our considerations in the proof of Theorem~\ref{main} show that $(\pi\times t)(\T^+_{X_G} )\simeq \A$ and the conclusion follows.
\end{proof}

Davidson and Roydor \cite{DR} have characterized when two tensor algebras of topological graphs over a zero-dimensional space are isomorphic. However the topological graphs they consider are compact, i.e., both the edge and vertex spaces are compact. Hence their theory does not apply here. Nevertheless, since our algebras are also triangular limit algebras we can capitalize on the existence of a complete invariant for triangular limit algebras.

Two topological graphs $G= (G^{0}, G^{1}, r , s)$ and $\fG= (\fG^{0}, \fG^{1}, r , s)$ are said to be \textit{conjugate} if there exist homeomorphisms $\phi_0: G^0\rightarrow \fG^0$ and $\phi_1: G^1\rightarrow \fG^1$ so that the following diagrams commute
\begin{equation} \label{comdiag}
\xymatrix{{G^1} \ar[r]^r \ar[d]_{\phi_1} & {G^0} \ar[d]^{\phi_0}  & &{G^1} \ar[r]^s \ar[d]_{\phi_1} & {G^0} \ar[d]^{\phi_0} \\
{\fG^1} \ar[r]^r  &{\fG^0} & &{\fG^1} \ar[r]^s  &{\fG^0}} 
\end{equation}
It is easy to see that if two topological graphs are conjugate then the associated tensor algebras are isometrically isomorphic. It follows from work of Davidson and Roydor \cite{DR} that the converse is also true provided that the topological graphs are compact and acting on totally disconnected spaces. Even though our topological graphs are not compact, this situation persists here.

\begin{corollary} \label{clas1}
Let $\A$ and $\fA$ be limit algebras which are also tensor algebras and let $G$ and $\fG$ be the topological graphs associated with $\R(\A)$ and $\R(\fA)$ respectively. If  $\A$ and $\fA$ are isometrically isomorphic as algebras then $G$ and $\fG$ are conjugate as topological graphs.
\end{corollary}

\begin{proof}
By \cite[Theorem 7.5]{Pow}, there exists a homeomorphism
\[
\theta: \D^{\star} \equiv G^0\longrightarrow \fG^0\equiv \fD^{\star}
\]
between the Gelfand spaces of the diagonals so that the map
\[
\theta^{(2)}: \D^{\star} \times \D^{\star} \ \longrightarrow  \fD^{\star} \times \fD^{\star} : (x, y) \longmapsto \big(\theta(x), \theta(y)\big)
\]
becomes a semigroupoid isomorphism between $\R(\A)$ and $\R(\fA)$,\break i.e., $\theta^{(2)}\big( \R(\A)\big) = \R(\fA)$ and $\theta^{(2)}$ preserves products. The preservation of products guarantees that the map 
\[
\delta(x, y)\equiv \delta_{\fA}\big( \theta(x), \theta(y)\big), \quad (x, y) \in\R(\A),
\]
is a continuous and coherent $\bbZ^+_0$-valued cocycle on $\R(\A)$. Hence, by Theorem~\ref{main1}, we have $\delta_{\A}= \delta$ and so $\theta^{(2)}(G^1) = \fG^{1}$. Therefore by taking $\phi_0=\theta$ and $\phi_1=\theta^{(2)}$ in (\ref{comdiag}) we obtain the desired conjugacy between $G$ and $\fG$.
\end{proof} 

We conclude the paper with the classification of the tree-refinement algebras of Example~\ref{tra}. Even though our techniques work on any UHF $\ca$-algebra, we choose to work with the $2^{\infty}$-UHF $\ca$-algebra for ease in notation. It turns out that the arguments involved in the proof of this classification are inductive in nature and require a variant of the well-known \textit{Aho-Hopcroft-Ullman algorithm} on the rooted-tree isomorphism problem \cite[pages 84--85]{AHU}, which we now describe.

An out-tree $G$ where each vertex, apart from sinks, emits at least two edges is said to be in \textit{reduced form}. Given an out-tree $G$, we associate a graph $\Gr$ in reduced form as follows. If $p, q\in G^0$ emit at least two edges and there is a path $e_1e_2\dots e_n=pe_1e_2\dots e_nq$ so that each one of the vertices $s(e_1), s(e_2), \dots s(e_{n-1})$ emits only one edge, then we delete these vertices from the graph, we write only one edge $e$ from $q$ to $p$ and we assign the weight $n-1$ to $p$. This way each edge of $\Gr$ apart from its sinks emits at least two edges. It is immediate that two out-trees $G$ and $\fG$ are isomorphic as graphs if and only if the associated graphs $\Gr$ and $\fG_{red}$ are isomorphic as in (\ref{comdiag}) via an isomorphism $\phi_0$ that preserves weights of vertices.

Given an out-tree $G$ in reduced form we define the \textit{height} of a vertex $p \in G^0$ as follows. All sinks of $G$ have height $0$; in general, an edge $p$ has height $k$ if one of its successors has height $k-1$ and all others have height $k-1$ or less. Given a graph $G$ in reduced form we create lists of weighted subgraphs of $G$ according to the following inductive scheme. The list $\L_0$ consists of all the sinks of $G$, including their weights. The list $\L_k$ consists of all the subtrees of $G$ that result from the list $\L_{k-1}$ by adding to the graphs of that list the appropriate vertices (and corresponding edges) of height $k$. Note that not all graphs in the list $\L_{k-1}$ are guaranteed to ``receive" a vertex of height $k$ when they ``move" to the list $\L_k$. Furthermore the list $\L_n$, where $n$ is the height of the root of $G$ is $G$ itself. 

The proof of the next theorem follows from the following elementary principle: in order to show that two out-trees $G$ and $\fG$ are isomorphic, it suffices to provide an argument that shows that at every level $k \in \bbN$, the corresponding lists $\L_k$ for the graphs $\Gr$ and $\fG_{red}$ ``coincide". The Aho-Hopcroft-Ullman algorithm is based on a less elaborate version of this principle as well.

\begin{theorem} \label{clas2}
Two tree-refinement subalgebras $\A\big(G, \{k_n\}\big)$ and $\A\big(\fG, \{l_n\}\big)$ of the $2^{\infty}$-UHF $\ca$-algebra are isometrically isomorphic if and only if the out-trees $G$ and $\fG$ are isomorphic, perhaps after an ampliation.
\end{theorem}

\begin{proof}
Since both $\A\big(G, \{k_n\}\big)$ and $\A\big(\fG, \{l_n\}\big)$ are assumed to be subalgebras of the $2^{\infty}$-UHF $\ca$-algebra, the sets $G^0$ and $\fG^0$ have cardinality a power of $2$. Furthermore, any tree-refinement embedding of multiplicity $2^n$ factors as a product of $n$ tree-refinement embeddings of multiplicity 2. Therefore we may assume that, perhaps after an ampliation, both graphs $G$ and $G_0$ have the same number of vertices, say $|G^0|=|\fG^0|=c$, and that $k_{n+1}/k_n = l_{n+1}/l_{n}=2$, for all $n \in \bbN$. From this it is immediate that if the out-trees are isomorphic then $\A\big(G, \{k_n\}\big)$ and $\A\big(\fG, \{l_n\}\big)$ are isometrically isomorphic as algebras.

Conversely, asume that $\A\big(G, \{k_n\}\big)$ and $\A\big(\fG, \{l_n\}\big)$ are isometrically isomorphic via an isomorphism $$\phi: \A\big(G, \{k_n\}\big) \longrightarrow \A\big(\fG, \{l_n\}\big).$$ 
Let $D$ and $\fD$ denote the diagonals of $\A\big(G, \{k_n\}\big)$ and $\A\big(\fG, \{l_n\}\big)$ with Gelfand spaces $D^{\star}$ and $\fD^{\star}$ respectively and let 
\[
\phi^{\star}:  D^{\star} \longrightarrow \fD^{\star}; x \longmapsto x\circ \phi^{-1}.
\]
Consider $S\subseteq D^{\star}$ which is maximal under the properties of being totally ordered and hereditary, i.e., if $x \in S$ and $y\leq x$, then $y \in S$. (Here $\leq$ denotes the order on $\D^{\star}$ arising from the semigroupoid $\R\big(A\big(G, \{k_n\}\big)\big)$.) It is easy to see that there exist $p, q\in G^0$ and a path $e_1e_2\dots e_n=pe_1e_2\dots e_nq$, with $|s^{-1}(q)|=0$, $|s^{-1}(p)|\geq 2$ and $|s(e_i)|=1$, $i=1, 2, \dots, n-1$, so that $S = \tilde{P}$, where $P=L_q + L_{s(e_1)}+\dots +L_{s(e_{n-1})}$. In other words $S = \tilde{P}$ corresponds to one of the sinks of $\Gr$. Furthermore  the weight of $S$ is equal to $m$, where $tr(P) = m/c$ and $tr$ is the unique normalized trace on the $2^{\infty}$-UHF $\ca$-algebra. Since both $(D^{\star}, \leq)$ and the normalized trace are preserved by $\phi$ and its induced maps, we have that $\phi^{\star}(S) \subseteq \fD^{\star}$ corresponds to a sink for $\fG_{red}$ with the same weight. Therefore $\phi^{\star}$ induces a bijection between the $\L_0$ lists of $\Gr$ and $\fG_{red}$ that preserves weights, i.e., a reduced graph isomorphism between the elements of that list.

By way of induction, assume that for $k \geq 1$, we have verified that $\phi^{\star}$ induces a bijection between the vertices of $\Gr$ and $\fG_{red}$ that preserves their heights up to level $k$ and also it induces a bijection between the lists $\L_0, \L_1, \dots, \L_k$ of $\Gr$ and $\fG_{red}$ that preserves weighted graph isomorphism classes. We show that the same happens for $k+1$.

Consider $S\subseteq D^{\star}$ which is maximal under the properties of being totally ordered and hereditary with respect to vertices of height at most $k$, i.e., if $x \in S$ and $y\leq x$, then $y \in S$ or $y \in S'$, where $S'$ is a vertex of height at most $k$. Arguing as before, it is easy to see that $S= \tilde{P}$, where $P$ corresponds to a vertex of height $k+1$ in $\Gr$ with weight $m$, where $tr(P) = m/c$. Furthermore, $S$ and its saturation in $\L_k$ determines an element of $\L_{k+1}$ which is mapped by $\phi^{\star}$ to an element in the $\L_{k+1}$ list for $\fG_{red}$ that belongs to the same isomorphism class as $S$ and its saturation. In particular $\phi^{\star}(S)$ has height $k+1$. This concludes the inductive step.

It is clear that the induction terminates successfully when the $\L_k$ lists consist of a single  element for both $\Gr$ and $\fG_{red}$. \end{proof}

The previous arguments can be modified to work in the general case as well. If $\A\big(G, \{k_n\}\big)$ and $\A\big(\fG, \{l_n\}\big)$ are isomorphic tree-refinement algebras, then by \cite[Theorem 7.5]{Pow}, both algebras have the same supernatural number. Since the tree-refinement embeddings commute, it follows that after perhaps two ampliations (one for each out-tree $G$ and $\fG$) we may assume that both $G^0$ and $\fG^0$ have the same cardinality. Repeating the arguments of the proof above, we obtain that the out-trees  $G$ and $\fG$ are isomorphic. Note that the ampliations for $G$ and $\fG$ cannot be arbitrary but instead compatible with the supernatural numbers of $\A\big(G, \{k_n\}\big)$ and $\A\big(\fG, \{l_n\}\big)$. 

Conversely, if $\A\big(G, \{k_n\}\big)$ and $\A\big(\fG, \{l_n\}\big)$ have the same supernatural number and the out-trees $G$ and $\fG$ are isomorphic after compatible ampliations, then one can show that $\A\big(G, \{k_n\}\big)$ and $\A\big(\fG, \{l_n\}\big)$ are isomorphic by using an argument similar to the zig-zag argument we use to show that two UHF $\ca$-algebras with the same supernatural number are isomorphic. We leave the details for the interested reader.


\end{document}